\documentclass{gtmon_a}
\pdfoutput=1

\usepackage{amscd}

\proceedingstitle{Proceedings of the Nishida Fest (Kinosaki 2003)}
\conferencestart{28 July 2003}
\conferenceend{8 August 2003}
\conferencename{International Conference in Homotopy Theory}
\conferencelocation{Kinosaki, Japan}

\editor{Matthew Ando}
\givenname{Matthew}
\surname{Ando}

\editor{Norihiko Minami}
\givenname{Norihiko}
\surname{Minami}

\editor{Jack Morava}
\givenname{Jack}
\surname{Morava}

\editor{W Stephen Wilson}
\givenname{W Stephen}
\surname{Wilson}

\title{Hypersurface representation and the image of\\the double $S^3$--transfer}

\author{Mitsunori Imaoka}
\givenname{Mitsunori}
\surname{Imaoka}
\address{Department of Mathematics Education\\
Graduate School of Education\\
Hiroshima University\\\newline
1--1--1 Kagamiyama\\
Higashi--Hiroshima 739-8524\\
Japan}
\email{imaoka@hiroshima-u.ac.jp}
\urladdr{}

\volumenumber{10}
\issuenumber{}
\publicationyear{2007}
\papernumber{10}
\startpage{187}
\endpage{193}

\doi{}
\MR{}
\Zbl{}

\arxivreference{}


\received{24 August 2004}
\revised{5 July 2005}
\accepted{}
\published{29 January 2007}
\publishedonline{29 January 2007}
\proposed{}
\seconded{}
\corresponding{}
\version{}

\makeatletter
\def\cnewtheorem#1[#2]#3{\newtheorem{#1}{#3}[section]
\expandafter\let\csname c@#1\endcsname\c@thm}

\let\xysavmatrix\xymatrix
\def\xymatrix{\disablesubscriptcorrection\xysavmatrix}
\AtBeginDocument{\let\bar\wbar\let\tilde\wtilde\let\hat\what}

\theoremstyle{plain}
\newtheorem{thm}{Theorem}
\cnewtheorem{prop}[thm]{Proposition}
\cnewtheorem{lem}[thm]{Lemma}
\cnewtheorem{cor}[thm]{Corollary}
\newtheorem*{thm*}{Theorem}

\theoremstyle{definition}
\cnewtheorem{defn}[thm]{Definition}
\cnewtheorem{exmp}[thm]{Example}
\cnewtheorem{prob}[thm]{Problem}

\theoremstyle{remark}
\cnewtheorem{rem}[thm]{Remark}

\makeatother  

\renewcommand{\H}{\mathbb H}

\begin{document}

\begin{abstract}
We study the image of a transfer homomorphism in the stable homotopy
groups of spheres.  Actually, we show that an element of order 8 in the 18
dimensional stable stem is in the image of a double transfer homomorphism,
which reproves a result due to P\,J Eccles that the element is represented
by a framed hypersurface.  Also, we determine the image of the transfer
homomorphism in the 16 dimensional stable stem.
\end{abstract}

\maketitle

\section{Introduction and result}
\label{sec1}

Let $\nu ^*\in\pi ^s_{18}(S^0)$ be an element of order $8$ in the 
stable homotopy groups of spheres.
Throughout the paper, $\pi _i(X)$ (resp. $\pi ^s_i(X)$) denotes the 
homotopy group (resp. the stable homotopy group) of a space $X$, 
and we use the notations of Toda~\cite{To} for elements of $\pi ^s_*(S^0)$. 
Then, using a generators $\nu\in\pi ^s_3(S^0)\cong \Z /24$ and 
$\sigma\in\pi ^s_7(S^0)\cong\Z /240$, 
$\nu ^*$ is represented by the Toda bracket 
$\langle\sigma ,2\sigma ,\nu \rangle=-\langle\sigma ,\nu ,\sigma \rangle$ 
with no indeterminacy.

As is known, $\nu ^*$ is not in the image of the homomorphism 
$J_0\co \pi _{18}(SO)\to \pi ^s_{18}(S^0)$ induced by the stable 
$J$--map $J\co SO\to \Omega^{\infty}\Sigma ^{\infty}S^0$.
We shall show that $\nu ^*$ is in the image of the homomorphism 
$J_1\co\pi _{19}(\Sigma SO)\to \pi ^s_{18}(S^0)$ induced by the adjoint 
map $\Sigma SO\to \Omega ^{\infty}\Sigma ^{\infty}S^1$ to $J$.
Actually, we prove that $\nu ^*$ is in a double 
$S^3$--transfer image which is a subgroup of $\mbox{Im}(J_1)$. 

Eccles \cite{Ec} has made clear that $\mbox{Im}(J_1)$ consists of
elements represented by framed hypersurfaces. 
Such study has also done by Rees \cite{Re}. 
Our result gives the following.

\begin{thm*}[Eccles~{\cite[page~168]{Ec}}]
$\nu ^*$ is 
representable by a framed hypersurface of dimension $18$.
\end{thm*}

Here, a framed hypersurface of dimension $18$ is a closed 
manifold of dimension 18 embedded in $S^{19}$ and framed in $S^{N}$ 
for sufficiently large $N$, 
and the theorem means that $\nu ^*$ is the framed cobordism class of 
such a framed hypersurface when we regard $\pi ^s_{18}(S^0)$ as 
the framed cobordism group $\Omega ^{fr}_{18}$.

The Bott map $B\co\Sigma ^3BSp\to SO$ is the adjoint map to the homotopy 
equivalence $BSp\times \Z\to \Omega ^3SO$, 
where $SO=\bigcup _n SO(n)$ is the rotation group and $BSp$ is 
the classifying space of the symplectic group $Sp=\bigcup _nSp(n)$.
Let $\H P^{\infty}=\bigcup _n\H P^n$ be the infinite dimensional 
quaternionic projective space. 
Then, $\H P^{\infty}=BSp(1)$, and there is an inclusion map 
$i\co\H P^{\infty}\to BSp$. 
We define a map $t_H$ as the composition
$$
t_H\co\Sigma ^3\H P^{\infty}\stackrel{i}{\longrightarrow}
\Sigma ^3BSp\stackrel{B}{\longrightarrow}SO\stackrel{J}{\longrightarrow}
\Omega ^{\infty}\Sigma ^{\infty}S^0.
$$
We denote by $t_H^s\co\H P^{\infty}\to S^{-3}$ a stable map adjoint to 
$t_H$. 
It is not certain whether $t_H^s$ is stably homotopic to the stable map 
constructed by the Boardman's transfer construction \cite{Bo} on the 
principal $S^3$--bundle over $\H P^{\infty}$.
However, since $t_H$ has the following property in common with the 
Boardman's $S^3$--transfer map, 
we call $t_H$ and $t_H^s$ the $S^3$--transfer maps.

\begin{lem}\label{Lem1}
The restriction 
$t_H' \co S^7\to \Omega ^{\infty}\Sigma ^{\infty}S^0$ 
of $t_H$ to the bottom sphere represents 
$\sigma\in\pi ^s_7(S^0)$ up to sign.
\end{lem}

Let $\mu \co \Omega ^{\infty}\Sigma ^{\infty}S^0\wedge\Omega ^{\infty}
\Sigma ^{\infty}S^0 \to \Omega ^{\infty}\Sigma ^{\infty}S^0$ 
be the multiplication on the infinite loop space 
$\Omega ^{\infty}\Sigma ^{\infty}S^0$ given by composition, and 
$\mu '\co \Sigma\Omega ^{\infty}\Sigma ^{\infty}S^0\wedge\Omega ^{\infty}
\Sigma ^{\infty}S^0 \to \Omega ^{\infty}\Sigma ^{\infty}S^1$ 
the adjoint map to $\mu$. 
Then, we set 
$$
t_H(2)= \mu '\circ \Sigma (t_H\wedge t_H)\co\Sigma ^7\H P^{\infty}
\wedge\H P^{\infty}\to\Omega ^{\infty}\Sigma ^{\infty}S^1, 
$$
and call $t_H(2)$ a double $S^3$--transfer map. 
Then, our main result may be stated as follows:

\begin{thm}\label{ThA}
There exists an element $a\in\pi _{12}(\H P^{\infty}\wedge\H P^{\infty})$ 
satisfying 
$$
t_H(2)_*(\Sigma ^7a)=\nu ^* \ \ \ in \ \ \pi _{18}^s(S^0). 
$$ 
\end{thm}

Applying a result due to Eccles and Walker~\cite{EW}, 
we have $\text{\rm Im}(t_H(2)_*) \subset\text{\rm Im}(J_1)$ 
(see \fullref{Prop2} and Corollary\ref{Cor3}).
Hence, it turns out that $\nu ^*$ is represented by a framed hypersurface 
of dimension $18$ as is stated in the first theorem.

In Carlisle et al~\cite{CE}, effective use of the $S^1$--transfer homomorphism 
$\tau \co\pi ^s_{i-1}(\C P^{\infty})\to\pi ^s_i(S^0)$ shows that 
certain elements are represented by framed hypersurfaces. 
In this respect, $\nu ^*$ is not in the image of the double 
$S^1$--transfer homomorphism 
$\tau _2\co \pi ^s_{16}(\C P^{\infty}\wedge
\C P^{\infty})\to\pi ^s_{18}(S^0)$. 
In fact, the image of $\tau _2$ in $\pi ^s_{18}(S^0)$ 
is equal to $\Z /4\{2\nu ^*\}$ by Imaoka~\cite[Theorem~10]{Im}.

By Toda~\cite{To}, 
$\pi _{16}^s(S^0)=\Z /2\{ \eta ^*\}\oplus \Z /2\{ \eta\circ\rho\}$. 
Here, $\eta\in\pi ^s_1(S^0)\cong\Z /2$ is a generator, 
$\eta ^*$ is an element in the Toda bracket 
$\langle\sigma ,2\sigma ,\eta \rangle$ 
and $\rho$ is a generator of the image of 
$J_0\co\pi _{15}(SO)\to \pi _{15}^s(S^0)$. 
Mahowald~\cite{Ma} has constructed an important family 
$\eta _i\in\pi _{2^i}^s(S^0)$ for $i\ge 3$, 
and $\eta ^*\equiv\eta _4 \pmod{\eta\circ\rho}$. 

By \fullref{Lem1}, we see the following:

\begin{lem}\label{tH20}
${\rm Im}(t_H(2)_*)=0$ in $\pi _{16}^s(S^0)$. 
\end{lem}

Mukai~\cite[Theorem~2]{Mu} has shown that both $\nu ^*$  and 
$\eta ^*$ are in the image of the $S^1$--transfer homomorphism 
$\tau \co\pi _{i-1}^s(\C P^{\infty})\to\pi _i^s(S^0)$.  
Morisugi~\cite[Corollary~E]{Mo} has shown that all Mahowald's 
elements are in the image of the $S^3$--transfer homomorphism 
$\pi _*^s(Q_{\infty})\to\pi _*^s(S^0)$ given for the 
quaternionic quasi-projective space $Q_{\infty}=\bigcup _nQ_n$. 

In contrast with these, we have the following:

\begin{thm}\label{PropB}
${\rm Im}((t_H^s)_*\co\pi _{13}^s(\H P^{\infty})
\to \pi _{16}^s(S^0))=\Z /2\{\eta\circ\rho\}$, 
and thus $\eta^*\not\in \text{\rm Im}(t_H^s)_*$.
\end{thm}

We also remark the following noted by Eccles~\cite{Ec}, where 
$J_2\co\pi _{18}(\Sigma ^2SO)\to \pi _{16}^s(S^0)$ is the 
homomorphism induced by the map 
$\Sigma ^2SO\to \Omega ^{\infty}\Sigma ^{\infty}S^2$ 
adjoint to the stable $J$--map. 

\begin{prop}\label{eta*2}
$\eta ^*\in{\rm Im}(J_2)$. 
\end{prop}

In \fullref{sec2} we prove \fullref{Lem1} 
and \fullref{ThA}, and in \fullref{sec3} we prove \fullref{tH20}, 
\fullref{PropB} and \fullref{eta*2}.

\section[Proof of \ref{ThA}]{Proof of \fullref{ThA}}
\label{sec2}

We first prove \fullref{Lem1}. 

\begin{proof}[\bf Proof of \fullref{Lem1}]
Let $\xi$ be the canonical left $\H$--line bundle over $\H P^{\infty}$, 
and put $\tilde{\xi}=\xi-1_{\H}$. 
Also, let $\tilde{\xi _1}^*$ be the adjoint to the 
restriction of $\tilde{\xi}$ to $\H P^1=S^4$. 
Then, the restriction of the classifying map $\varphi$ of 
$\tilde{\xi}_1^*\otimes _{\H}\tilde{\xi}$ to the bottom sphere 
$S^8=S^4\wedge S^4$ represents a generator of $\pi _8(BSO)\cong\Z$.
As is known, the adjoint map $\Sigma ^4\H P^{\infty}\to BSO$ to the 
composition $B\circ i\co\Sigma ^3\H P^{\infty}\to SO$ 
is homotopic to $\varphi$. 
Since $\sigma =J_0(\iota _7)$ for a generator $\iota _7\in\pi _7(SO)$, 
$t_H'$ represents $\sigma$ up to sign. 
\end{proof}

Let $H(m)\co\Sigma SO\wedge SO\to\Sigma SO$ be the map defined by the 
Hopf construction on the multiplication 
$m\co SO\times SO\to SO$, and 
$J^*\co\Sigma SO\to \Omega ^{\infty}\Sigma ^{\infty}S^1$ the adjoint map 
to the stable $J$--map $J\co SO\to\Omega ^{\infty}\Sigma ^{\infty}S^0$. 
Then, Eccles and Walker have shown the following \cite[Proposition~2.2]{EW}, 
in which 
$\mu '\co\Sigma (\Omega ^{\infty}\Sigma ^{\infty}S^0\wedge
\Omega ^{\infty}\Sigma ^{\infty}S^0)\to \Omega ^{\infty}
\Sigma ^{\infty}S^1$ 
is the adjoint map to the multiplication 
$\mu \co\Omega ^{\infty}\Sigma ^{\infty}S^0\wedge\Omega ^{\infty}
\Sigma ^{\infty}S^0\to \Omega ^{\infty}\Sigma ^{\infty}S^0$. 

\begin{prop}\label{Prop2}
The composition 
$J^*\circ H(m)\co\Sigma SO\wedge SO\to\Omega ^{\infty}
\Sigma ^{\infty}S^1$ 
is homotopic to $\mu '\circ\Sigma (J\wedge J)$.
\end{prop}

Recall that the homomorphism $J_1\co\pi _i(\Sigma SO)\to\pi ^s_{i-1}(S^0)$ 
is induced by $J^*$. 
Since $(J\wedge J)\circ (B\circ i\wedge B\circ i)$ is equal to 
$t_H\wedge t_H$,
by composing it with the maps of \fullref{Prop2}, 
$J^*\circ H(m)\circ \Sigma (B\circ i\wedge B\circ i)
\co\Sigma (\Sigma ^3\H P^{\infty}\wedge\Sigma ^3\H P^{\infty})\to
\Omega ^{\infty}\Sigma ^{\infty}S^1$ is homotopic to $t_H(2)=
\mu '\circ \Sigma (t_H\wedge t_H)$. 
Hence, we have the following.

\begin{cor}\label{Cor3}
$\text{\rm Im}(t_H(2)_*)\subset\text{\rm Im}(J_1)$ in 
$\pi _i^s(S^0)$.
\end{cor}

Let $\beta _i\in H_{4i}(\H P^{\infty};\thinspace\Z )\cong\Z$ 
be a generator.
Then, we have 
$H_{12}(\H P^{\infty}\wedge\H P^{\infty};\thinspace\Z )=
\Z\{\beta _1\otimes\beta _2\}\oplus\Z\{\beta _2\otimes\beta _1\}$.

\begin{lem}\label{Lem4}
There exists an element $a\in\pi _{12}(\H P^{\infty}\wedge\H P^{\infty})$
satisfying $h(a)=\beta _1\otimes\beta _2-\beta _2\otimes\beta _1$ for
the Hurewicz homomorphism $h$.
\end{lem}

\begin{proof} 
Notice that $\pi _{12}(\H P^{\infty}\wedge\H P^{\infty})
\cong \pi ^s_{12}(\H P^{\infty}\wedge\H P^{\infty})$. 
Since $\H P^2=S^4\cup _{\nu}e^8$, the $12$--skeleton of 
$\H P^{\infty}\wedge\H P^{\infty}$ has a cell structure 
$S^8\cup _{\nu\vee\nu}(e^{12}\cup e^{12})$.
Thus, we have the following exact sequence:
$$
\pi ^s_{12}(S^8)=0\to \pi _{12}^s(\H P^{\infty}\wedge
\H P^{\infty})\stackrel{\psi}{\longrightarrow}\pi ^s_{12}
(S^{12}\vee S^{12})\stackrel{\varphi}{\longrightarrow}\pi ^s_{11}(S^8).
$$
Here, $\varphi \co\pi ^s_{12}(S^{12}\vee S^{12})=\Z\{ \iota _1\}\oplus 
\Z\{ \iota _2\}\to \pi ^s_{11}(S^8)=\Z /24\{\nu\}$ satisfies 
$\varphi (\iota _1)=\varphi (\iota _2)=\nu$.
Therefore, 
$\pi ^s_{12}(\H P^{\infty}\wedge\H P^{\infty})=\Z\{ a\}\oplus\Z
\{ b\}$ for some elements $a$ and $b$ satisfying 
$\psi _*(a)=\iota _1-\iota _2$ and $\psi _*(b)=24\iota _2$.
Then, $a$ is a required element up to sign, since we have 
$\{\psi _*(\beta _1\otimes\beta _2), \psi _*(\beta _2\otimes\beta _1)\}
=\{ h(\iota _1),h(\iota _2)\}$ up to sign. 
\end{proof}

\begin{proof}[Proof of \fullref{ThA}]
Let $a\in\pi _{12}(\H P^{\infty}\wedge\H P^{\infty})$ be the element 
in \fullref{Lem4}. 
Since $\pi _{12}(\H P^{\infty}\wedge\H P^{\infty})\cong
\pi _{12}^s(\H P^2\wedge\H P^2)$ and 
$t_H(2)_*=(t_H^s\wedge t_H^s)_*\Sigma ^{\infty}$ 
for the suspension isomorphism 
$\Sigma ^{\infty}\co\pi _{19}(\Sigma ^7\H P^2\wedge\H P^2)\to
\pi ^s_{12}(\H P^2\wedge\H P^2)$, 
we identify $a$ with 
$\Sigma ^{\infty}a\in\pi ^s_{12}(\H P^2\wedge\H P^2)$. 
Then, it is sufficient to show $(t_H^s\wedge t_H^s)_*(a)=\nu ^*$ 
up to sign.

The following diagram is stably homotopy commutative up to 
sign:
$$
\begin{CD}
S^{12}@>a>>\H P^2\wedge\H P^2@>t_H^s\wedge 1>>S^{-3}\wedge\H P^2@>1
\wedge t_H^s>>S^{-6} \\
@VV\simeq V  @VV1\wedge qV @VVqV \\
S^4\wedge S^8@>i\wedge 1>>\H P^2\wedge S^8@>t_H^s\wedge 1>>S^{-3}\wedge S^8, 
\end{CD}
$$
where $q$ is the collapsing map to the top cell. 
The left hand square is stably homotopy commutative because 
$(1\wedge q)_*(a)$ is a generator of $\pi ^s_4(\H P^2)$ by \fullref{Lem4}.
Hence, $q\circ (t_H^s\wedge 1)\circ a\simeq (t_H^s\circ i)
\wedge 1\simeq \sigma$ by
\fullref{Lem1}, and thus $(t_H^s\wedge 1)\circ a$ is a coextension of 
$\sigma$.
On the other hand, 
since $t_H^s\co\H P^2\to S^{-3}$ is an extension of 
$\sigma \co S^{4}\to S^{-3}$ by \fullref{Lem1}, 
$$
(t_H^s\wedge t_H^s)_*(a)=(1\wedge t_H^s)_*(t_H^s\wedge 1)_*(a)
\in\langle\sigma ,\nu ,\sigma \rangle =\{ -\nu ^*\}
$$ 
by definition of the Toda bracket.
Thus, we have completed the proof.
\end{proof}

\section{Transfer image in $\pi ^s_{16}(S^0)$}
\label{sec3}

We remark that 
$$
\pi _{17}(\Sigma ^7\H P^{\infty}\wedge \H P^{\infty})\cong 
\pi ^s_{10}(\H P^{\infty}\wedge \H P^{\infty})\cong 
\pi ^s_{10}(S^8)=\Z /2\{\eta ^2\}. 
$$
Then, using \fullref{Lem1}, 
$$
t_H(2)_*(\eta ^2)=((t_H^s)_*(\eta ))^2=(\sigma\circ\eta )^2=0, 
$$
and thus we have \fullref{tH20}. 

Lastly, we prove \fullref{PropB} and \fullref{eta*2}. 

\begin{proof}[\bf Proof of \fullref{PropB}]
As mentioned in \fullref{sec1}, 
$\pi _{16}^s(S^0)=\Z /2\{ \eta ^*\}\oplus \Z /2\{\eta\circ
\rho\}$ by Toda~\cite{To}, and we shall show 
$$
\text{\rm Im}
((t_H^s)_*\co\pi _{13}^s(\H P^{\infty})\to \pi _{16}^s(S^0))=
\Z /2\{\eta\circ\rho\}.
$$

Notice that $\pi _{13}^s(\H P^{\infty})\cong\pi _{13}^s(\H P^{3})$ 
and the attaching map $\varphi \co S^{11}\to \H P^2$ of 
the top cell in $\H P^3$ is stably a coextension of 
$2\nu\in\pi _3^s(S^0)$. 
We consider the following exact sequence:
$$
\pi _{14}^s(S^{12})\stackrel{\varphi _*}{\longrightarrow}
\pi _{13}^s(\H P^2)\stackrel{i_*}{\longrightarrow} \pi _{13}^s(\H P^3)
\stackrel{q_*}{\longrightarrow} \pi _{13}^s(S^{12})\stackrel{\varphi _*}
{\longrightarrow}\pi _{12}^s(\H P^2),
$$
where $i$ is the inclusion $\H P^2\subset \H P^3$ and 
$q$ is the collapsimg map $\H P^3\to S^{12}$ to the top cell.

Since $\H P^2\approx S^4\cup _{\nu}e^8$, we have 
\begin{align*}
\pi _{12}^s(\H P^2)&=\Z /2\{(i_0)_*(\overline{\nu} )\}\oplus\Z /2
\{(i_0)_*(\epsilon )\}\\
\text{and}\quad
\pi _{13}^s(\H P^2)&=\Z /2\{(i_0)_*(\mu )\}
\oplus\Z /2\{(i_0)_*(\eta\circ\epsilon )\}
\end{align*}
for the bottom inclusion $i_0\co S^4\to \H P^2$, 
where
$$\pi _8^s(S^0)=\Z /2\{\overline{\nu}\}\oplus\Z /2\{\epsilon\}
\quad\text{and}\quad
\pi _9^s(S^0)=\Z /2\{\mu\}\oplus\Z /2\{\eta\circ\epsilon\}\oplus
\Z /2\{\nu ^3\}.$$
Since $\varphi$ is a coextension of $2\nu$, 
we have 
$\varphi _*(\eta )=\varphi\circ\eta\in (i_0)_*
(\langle\nu , 2\nu , \eta \rangle)\equiv (i_0)_*(\epsilon )
\pmod{i_*(\overline{\nu})}$. 
Thus, $\varphi _*\co\pi _{13}^s(S^{12})\to \pi _{12}^s(\H P^2)$ 
is a monomorphism. 
Similarly, $\varphi _*(\eta ^2)=\varphi\circ\eta ^2$ is 
contained in 
$(i_0)_*(\langle\nu ,2\nu ,\eta ^2\rangle)\equiv (i_0)_*
(\eta\circ\epsilon )\pmod{(i_0)_*(\nu ^3)}$. 
Hence, we have 
$\pi _{13}^s(\H P^{\infty})\cong\pi _{13}^s(\H P^{3})
=\Z /2\{(i_0)_*(\mu )\}$. 

Since $t_H^s\co\Sigma ^3\H P^{\infty}\to S^0$ is an extension of 
$\sigma\co S^7\to S^0$ by \fullref{Lem1}, 
we conclude that 
Im$((t_H^s)_*\co\pi _{13}^s(\H P^{\infty})\to \pi _{16}^s(S^0))
=\Z /2\{\sigma\circ\mu\}=\Z /2\{\eta\circ\rho\}$.
\end{proof}

\begin{proof}[Proof of \fullref{eta*2}]
As mentioned in \fullref{sec1}, 
$$
\eta ^*\in \langle\sigma ,2\sigma ,\eta\rangle\subset 
\pi _{16}^s(S^0), 
$$
and the indeterminacy of $\langle\sigma ,2\sigma ,\eta\rangle$ is 
$\Z /2\{\eta\circ\rho\}$. 

Recall that $\sigma=J_0(\iota _7)\in\pi _7^s(S^0)$ up to sign 
for a generator $\iota _7\in\pi _7(SO)\cong\Z$. 
By Toda~\cite[Proposition~5.15, (5.16), Proposition~5.1]{To}, there
exist elements $\sigma '\in\pi _{14}(S^7)$ and $\eta _{16}\in
\pi_{17}(S^{16})$ which suspend to $2\sigma$ and $\eta$, respectively.
Then, $\iota _7\circ\sigma '=0$ since $\pi _{14}(SO)=0$.
Also, by \cite[Theorem~7.1]{To}, $(\Sigma ^2\sigma ')\circ\eta _{16}=0$. 
Hence, an unstable Toda bracket 
$\{ \Sigma ^2\iota _7,\Sigma ^2\sigma ',\eta _{16}\}$ 
is defined in $\pi _{18}(\Sigma ^2SO)$, 
and it satisfies 
$$
J_2(\{ \Sigma ^2\iota _7,\Sigma ^2\sigma ',\eta _{16}\})
\cap
\langle\sigma ,2\sigma ,\eta \rangle\neq\phi .
$$ 
Hence, $\eta ^*$ or $\eta ^*+\eta\circ\rho$ is in ${\rm Im}(J_2)$. 
But, since $\eta\circ\rho\in {\rm Im}(J_0)$ and 
${\rm Im}(J_0)\subset {\rm Im}(J_2)$, we have the required result. 
\end{proof}

\bibliographystyle{gtart}
\bibliography{link}

\begin{thebibliography}{}
\providecommand\bibmarginpar{\leavevmode\marginpar}
\def\urlstyle#1{{\tt #1}}

\bibitem{Bo}
\textbf{J\,M Boardman}, \emph{Stable homotopy theory}, mimeographed notes,
  University of Warwick (1966)

\bibitem{CE}
\textbf{D Carlisle}, \textbf{P Eccles}, \textbf{S Hilditch}, \textbf{N Ray},
  \textbf{L Schwartz}, \textbf{G Walker}, \textbf{R Wood},
  \href{http://dx.doi.org/10.1007/BF01175047} {\emph{Modular representations of
  $\mathrm{GL}(n,p)$, splitting
  $\Sigma(\mathbb{C}P^{\infty}\times\cdots\times\mathbb{C}P^{\infty})$, and the
  $\beta$--family as framed hypersurfaces}}, Math. Z. 189 (1985) 239--261
  \xox{MR}{779220}

\bibitem{Ec}
\textbf{P\,J Eccles}, \emph{Filtering framed bordism by embedding codimension},
  J. London Math. Soc. $(2)$ 19 (1979) 163--169 \xox{MR}{527750}

\bibitem{EW}
\textbf{P\,J Eccles}, \textbf{G Walker}, \emph{The elements {$\beta_{1}$} are
  representable as framed hypersurfaces}, J. London Math. Soc. $(2)$ 22 (1980)
  153--160 \xox{MR}{579819}

\bibitem{Im}
\textbf{M Imaoka}, \href{http://dx.doi.org/10.1016/0166-8641(96)00028-4}
  {\emph{The double complex transfer at the prime {$2$}}}, Topology Appl. 72
  (1996) 199--207 \xox{MR}{1406309}

\bibitem{Ma}
\textbf{M Mahowald}, \emph{A new infinite family in ${~}_{2}\pi_{*}^{s}$},
  Topology 16 (1977) 249--256 \xox{MR}{0445498}

\bibitem{Mo}
\textbf{K Morisugi}, \emph{Lifting problem of {$\eta$} and {M}ahowald's element
  {$\eta\sb j$}}, Publ. Res. Inst. Math. Sci. 25 (1989) 407--414
  \xox{MR}{1018509}

\bibitem{Mu}
\textbf{J Mukai}, \emph{On stable homotopy of the complex projective space},
  Japan. J. Math. $($N.S.$)$ 19 (1993) 191--216 \xox{MR}{1231514}

\bibitem{Re}
\textbf{E Rees}, \emph{Framings on hypersurfaces}, J. London Math. Soc. $(2)$
  22 (1980) 161--167 \xox{MR}{579820}

\bibitem{To}
\textbf{H Toda}, \emph{Composition methods in homotopy groups of spheres},
  Annals of Mathematics Studies 49, Princeton University Press, Princeton, N.J.
  (1962) \xox{MR}{0143217}

\end{thebibliography}

\end{document}